\documentclass[pdflatex,sn-mathphys-num]{sn-jnl}% Math and Physical Sciences Numbered Reference Style 
\usepackage{graphicx}%
\usepackage{multirow}%
\usepackage{amsmath,amssymb,amsfonts}%
\usepackage{amsthm}%
\usepackage{mathrsfs}%
\usepackage[title]{appendix}%
\usepackage{xcolor}%
\usepackage{textcomp}%
\usepackage{manyfoot}%
\usepackage{booktabs}%
\usepackage{algorithm}%
\usepackage{algorithmicx}%
\usepackage{algpseudocode}%
\usepackage{listings}%
\newtheorem{theorem}{Theorem}[section]
\newtheorem{lemma}[theorem]{Lemma}
\newtheorem{proposition}[theorem]{Proposition}
\newtheorem{corollary}[theorem]{Corollary}
\theoremstyle{definition}

\newtheorem{definition}[theorem]{Definition}

\numberwithin{equation}{section}

\allowdisplaybreaks
\begin{document}
	
	\title[Article Title]{A Durrmeyer-type variant of Gr\"unwald
		Interpolation Operators}
	
	
	\author*[]{\fnm{P. C.} \sur{Vinaya}}\email{vinayapc01@gmail.com}

	\affil*[1]{\orgdiv{Department of Mathematics}, \orgname{Cochin University of Science and Technology}, \orgaddress{\street{Kalamassery}, \city{Cochin-22}, \postcode{682022}, \state{Kerala}, \country{India}}}
	
	
	\abstract{
		\begin{abstract}
			
			In this paper, we construct a Durrmeyer-type variant of Gr\"unwald interpolation operators on the space $L^p[0,\pi]$. We prove their fundamental properties, including boundedness and convergence in the $L^p$-norm. We establish the convergence results using a Korovkin-type theorem in the setting of Banach function spaces. Furthermore, we obtain quantitative estimates for the convergence by means of the modulus of continuity and an appropriate $K$-functional.
		\end{abstract}
	}

	
	

	
	
	\maketitle
\section{Introduction}
Many sequences of operators defined on $C[0,1]$, including Bernstein operators, Baskakov operators, and Stancu operators, do not, in general, possess approximation properties in the space $L^1[0,1]$. In 1930, L.~Kantorovich introduced an integral modification of the Bernstein polynomials to overcome this limitation \cite{kantorovich}. As a consequence, the convergence of such operators can be extended to the more general space $L^p[0,1]$, for $1 \leq p < +\infty$. A more general integral modification was later proposed by J.~L.~Durrmeyer in 1967 (see \cite{durrmeyer}). Since then, several Durrmeyer-type modifications of various operators have been introduced, and their approximation properties have been extensively studied. A comprehensive survey of these modifications can be found in \cite{gupta2016survey}. A few of these are given in \cite{durrmeyervar, durrmeyervar1, prakash2023}. This work disusses a Durrmeyer-type variant of the classical Gr\"unwald interpolation operators. A Kantorovkich-type variant of Gr\"unwald interpolation operators and their approximation properties are obtained in \cite{GK}. We establish convergence properties of these operators in $L^p[0,\pi]$ and obtain the rate of convergence using modulus of continuity and a suitable K-functional. First, we recall the definition of the Gr\"unwald interpolation operators on Chebyshev nodes. For this purpose, we begin with the definition of the Lagrange interpolation operators. These operators generate the unique polynomial of degree at most $n-1$ that interpolates a given continuous function on $[-1,1]$ on a set of $n$ distinct nodes. 

The Chebyshev nodes of the first kind on $[-1,1]$ are given by
\[
x_k = \cos\big(\theta_k^{(n)}\big), \ \text{where } \theta_k^{(n)} = \frac{2k-1}{2n}\pi,\; k=1,2,\dots,n.
\]
The Lagrange interpolation operator based on Chebyshev nodes can be expressed as
\[
L_n(f)(t) = \sum_{k=1}^n f(\cos \theta_k^{(n)})(-1)^{k+1} \frac{\cos(nt) \sin \theta_k^{(n)}}{n(\cos t - \cos \theta_k^{(n)})},
\]
where
\[
P_k(t) = (-1)^{k+1} \frac{\cos(nt) \sin \theta_k^{(n)}}{n(\cos t - \cos \theta_k^{(n)})},
\quad k = 1, 2, \dots, n.
\]

It is well known that the Lagrange interpolation operators on Chebyshev nodes does not converge (see \cite{faber, mills, grunwald1, mills, mills1, josef}).

In 1941, G.~Gr\"unwald established a convergence result for an averaged sequence of operators constructed using the Lagrange interpolation operators on Chebyshev nodes. 
He obtained the following theorem.

\begin{theorem}[\textbf{G. Grünwald}, \cite{grunwald}]\label{grun}
	Let $f\in C[-1,1]$, then 
	\[
	\lim\limits_{n\rightarrow\infty}\frac{1}{2}\{L_n(f)(\theta-\frac{\pi}{2n})+L_n(f)(\theta+\frac{\pi}{2n})\}=f(\cos(\theta)),
	\]
	where the convergence is uniform on the interval $[0,\pi]$.
\end{theorem}
The following lemma plays an important role in obtaining the theorem.
\begin{lemma}[\textbf{G. Grünwald}, \cite{grunwald}]\label{lem}
	\[
	\frac{1}{2}\sum\limits_{k=1}^n \left| P_k\left(\theta - \frac{\pi}{2n}\right) + P_k\left(\theta + \frac{\pi}{2n}\right) \right| < c_1,
	\]
	where $c_1>0$ is an absolute constant
\end{lemma} 
Grünwald considered the following sequence of operators.
\begin{definition}\label{def}
	For $n=1,2,\ldots$, define $G_n:C[0,\pi]\rightarrow C[0,\pi]$ by 
	\[
	G_n(f)(\theta):=\frac{1}{2}\sum\limits_{k=1}^nf(\theta_k^{(n)})\{P_k(\theta-\frac{\pi}{2n})+P_k(\theta+\frac{\pi}{2n})\}.
	\]
\end{definition}
We refer to these interpolation operators as Gr\"unwald interpolation operators, or simply Gr\"unwald operators. These operators also satisfies, $\lim\limits_{n\to\infty}\|G_n(f)-f\|_\infty=0$ for all $f\in C[0,\pi]$.
 
In this article, we introduce a Durrmeyer-type variant of the Gr\"unwald interpolation operators on the space $L^p[0,\pi]$ for $1\leq p<+\infty$. Since the Lagrange interpolation operators are non-positive, so are the Gr\"unwald interpolation operators. It follows from the construction that the new sequence of operators is also non-positive. We first show that the operators are bounded in the space $C[0,\pi]$ and establish their convergence in this space. Moreover, we obtain a quantitative estimate for the rate of convergence using the modulus of continuity. We also prove the boundedness and convergence of these operators in the general space $L^p[0,\pi]$. For this purpose, we employ a fundamental interpolation theorem from harmonic analysis, namely the Riesz--Thorin interpolation theorem. We establish the convergence results in $L^p[0,\pi]$ by means of a non-positive Korovkin-type theorem given in \cite{vinaya}, which extends the corresponding positive version in \cite{zeren}.

This article is organized as follows. In the next section, we introduce a new sequence of operators on $L^p[0,\pi]$. Restricting our attention to the space $C[0,\pi]$, we establish the boundedness, convergence, and rate of convergence of these operators. In the subsequent section, we prove the boundedness and convergence of these operators in the space $L^p[0,\pi]$ for $1 \leq p < +\infty$. Using a suitable $K$-functional, we also obtain the rate of convergence in $L^p[0,\pi]$.

\section{Gr\"unwald--Durrmeyer Operators}
In this section, we introduce a new sequence of operators on $L^p[0,\pi]$ for $1\leq p\leq +\infty$. The operator we introduce is a Durrmeyer-type variant of the Gr\"unwald interpolation operators, similar to the Durrmeyer-type variant of the Bernstein operators. The sequence of operators are defined as follows.
\begin{definition}
	 Let 
	\[
	S_{k,n}(t)
	:=
	\frac{1}{2}
	\left(
	P_k\!\left(t - \frac{\pi}{2n}\right)
	+
	P_k\!\left(t + \frac{\pi}{2n}\right)
	\right),
	\quad k = 1,2,\dots,n.
	\]. 
	For $1\leq p\leq+\infty$, the sequence of operators $\mathcal{D}_n: L^p[0,\pi]\to L^p[0,\pi]$ is defined as
	\begin{equation}\label{durr}
		\mathcal{D}_n(f)(\theta)=\frac{n}{\pi}\sum\limits_{k=1}^nS_{k,n}(\theta)\int\limits_{0}^\pi f(t)S_{k,n}(t)\ dt
	\end{equation}
\end{definition}
We refer to these operators as Gr\"unwald--Durrmeyer operators. We show that $\mathcal{D}_n(1)=1$ for $n=1,2,\ldots$. 
 \begin{proposition}
 	$\mathcal{D}_n(1)=1$ for $n=1,2,\ldots$.
 \end{proposition}
 \begin{proof}
 	First, we claim that $\displaystyle{
 		\int\limits_{0}^\pi S_{k,n}(t)\ dt=\frac{\pi}{n}.
 		}$ To prove this, let $t\in [0,\pi]$. By the definition of $S_{k,n}$, we have 
 		\[
 		S_{k,n}(t)
 		=
 		\frac{(-1)^{k+1}\sin \theta_k^{(n)}}{2n}
 		\left[
 		\frac{\cos\!\left(n\left(t - \frac{\pi}{2n}\right)\right)}
 		{\cos\!\left(t - \frac{\pi}{2n}\right) - \cos \theta_k^{(n)}}
 		+
 		\frac{\cos\!\left(n\left(t + \frac{\pi}{2n}\right)\right)}
 		{\cos\!\left(t + \frac{\pi}{2n}\right) - \cos \theta_k^{(n)}}
 		\right].
 		\]
 		Let $\theta\in [0,\pi]$, with $\theta\neq t$. Define
 		\[
 		H_n(t,\theta)
 		=
 		\frac{\sin(n\theta)\,\sin \theta}{2n}
 		\left[
 		\frac{\cos\!\left(n\left(t - \frac{\pi}{2n}\right)\right)}
 		{\cos\!\left(t - \frac{\pi}{2n}\right) - \cos \theta}
 		+
 		\frac{\cos\!\left(n\left(t + \frac{\pi}{2n}\right)\right)}
 		{\cos\!\left(t + \frac{\pi}{2n}\right) - \cos \theta}
 		\right].
 		\]
 		We compute
 		\[
 		I := \int\limits_0^\pi H_n(t,\theta)\,dt.
 		\]
 		Let $\alpha = \frac{\pi}{2n}$. Then $I = \frac{\sin(n\theta)\sin\theta}{2n}(I_1 + I_2),$
 		where
 		\[
 		I_1 = \int\limits_0^\pi \frac{\cos(n(t-\alpha))}{\cos(t-\alpha)-\cos\theta}\,dt,\quad I_2 = \int\limits_0^\pi \frac{\cos(n(t+\alpha))}{\cos(t+\alpha)-\cos\theta}\,dt.
 		\]
 		For $I_1$, let $u = t - \alpha$. Then, we get
 		$I_1 = \displaystyle{\int\limits_{-\alpha}^{\pi-\alpha} \frac{\cos(nu)}{\cos u - \cos\theta}\,du.}$
 		For $I_2$, let $u = t + \alpha$. So that,
 		$I_2 =\displaystyle{\int\limits_{\alpha}^{\pi+\alpha} \frac{\cos(nu)}{\cos u - \cos\theta}\,du.}$
 		
 		Using periodicity of $\cos(nu)$ and the kernel, both reduce to
 		\[
 		I_1 = I_2 = \int\limits_0^\pi \frac{\cos(nu)}{\cos u - \cos\theta}\,du.
 		\]
 		We use the following known identity.
 		\[
 		\int\limits_0^\pi \frac{\cos(nu)}{\cos u - \cos\theta}\,du
 		=
 		\frac{\pi \sin(n\theta)}{\sin\theta},\ \theta\neq 0,\pi.
 		\]
 		Hence, we have 
 		\[
 		I = \frac{\sin(n\theta)\sin\theta}{2n}
 		\cdot(I_1+I_2)=\frac{\sin(n\theta)\sin\theta}{2n}
 		\cdot
 		2 \cdot \frac{\pi \sin(n\theta)}{\sin\theta}.
 		\]
 		
 		Canceling $\sin\theta$ ($\theta\neq 0,\pi$), we obtain
 		$I = \frac{\pi}{n}\sin^2(n\theta).$
 		Finally, we have 
 		\[
 		\int\limits_0^\pi H_n(t,\theta)\,dt
 		=
 		\frac{\pi}{n}\sin^2(n\theta).
 		\]
 		Now, we obtain
 		\[
 		\int\limits_{0}^\pi S_{k,n}(t)\ dt=\int\limits_0^\pi H_n\!\big(t,\theta_k^{(n)}\big)\,dt
 		=
 		\frac{\pi}{n}.
 		\]
 	By the Definition (\ref{durr}), we have
 	\begin{equation*}
 		\mathcal{D}_n(f)(\theta)=\frac{n}{\pi}\sum\limits_{k=1}^nS_{k,n}(\theta)\int\limits_{0}^\pi S_{k,n}(t)\ dt.
 	\end{equation*}
 	Since $\sum\limits_{k=1}^nS_{k,n}(\theta)=1$ for all $\theta\in [0,\pi]$, the result follows.
 \end{proof}
 \section{Convergence in the space $C[0,\pi]$}
 In this section, we prove the uniform boundedness and convergence of the sequence $\{\mathcal{D}_n\}_{n\in\mathbb{N}}$ in the space $C[0,\pi]$. Moreover, we obatain a quantitative estimate for the convergence of this sequence using the modulus of continuity. We begin by establishing the boundedness property.
\begin{lemma}\label{bdd}
	The sequence $\{\mathcal{D}_n\}_{n\in\mathbb{N}}$ is bounded uniformly in $C[0,\pi]$.
\end{lemma}
\begin{proof}
	\begin{align*}
		|\mathcal{D}_n(f)(\theta)|&\leq n\sum\limits_{k=1}^n|S_{k,n}(\theta)|\int\limits_{0}^\pi |f(t)||S_{k,n}(t)|\ dt\\
		&\leq n\|f\|_\infty\sum\limits_{k=1}^n|S_{k,n}(\theta)|\int\limits_{0}^\pi|S_{k,n}(t)|\ dt
	\end{align*}
	From Proposition $3.1$ in \cite{GK}, we have $\int\limits_{0}^\pi|S_{k,n}(t)|\ dt\leq \frac{M}{n}$ for some absolute constant $M>0$. Using this fact and by Lemma \ref{lem}, we have $|\mathcal{D}_n(f)(\theta)|\leq C\|f\|_\infty$ for some constant $C>0$ which does not depend on $n, f$ or $\theta$. Taking supremum over all $\theta\in [0,\pi]$, we have the result.
\end{proof}
The following theorem establishes a convergence result of the sequence of operators $\{\mathcal{D}_n\}_{n\in\mathbb{N}}$ in the space $C[0,\pi]$.
\begin{theorem}\label{main}
	Let $f\in C[0,\pi]$. Then,
	\[
	\lim\limits_{n\to\infty}\|\mathcal{D}_n(f)-f\|_\infty=0
	\]
\end{theorem}
\begin{proof}
	Let $\theta\in [0,\pi]$ and $f\in C[0,\pi]$. Then, $f(\theta)$ can be expressed as 
	\[
	f(\theta)=\frac{n}{\pi}\sum\limits_{k=1}^nS_{k,n}(\theta)\int\limits_{0}^\pi f(\theta)S_{k,n}(t)\ dt.
	\]
	Now, we have
	\begin{align*}
		|\mathcal{D}_n(f)(\theta)-f(\theta)|&\leq |\frac{n}{\pi}\sum\limits_{k=1}^nS_{k,n}(\theta)\int\limits_{0}^\pi f(t)S_{k,n}(t)\ dt-f(\theta)|\\
		&\leq |\frac{n}{\pi}\sum\limits_{k=1}^nS_{k,n}(\theta)\int\limits_{0}^\pi f(t)-f(\theta_k^{(n)})S_{k,n}(t)\ dt|+\\
		&|\frac{n}{\pi}\sum\limits_{k=1}^nS_{k,n}(\theta)\int\limits_{0}^\pi f(\theta_k^{(n)})-f(\theta)S_{k,n}(t)\ dt|\\
		&\leq \frac{n}{\pi}\sum\limits_{k=1}^n|S_{k,n}(\theta)|\int\limits_{0}^\pi |f(t)-f(\theta_k^{(n)})||S_{k,n}(t)|\ dt+\\
		&\frac{n}{\pi}\sum\limits_{k=1}^n|f(\theta_k^{(n)})-f(\theta)|S_{k,n}(\theta)||\int\limits_{0}^\pi S_{k,n}(t)\ dt|\\
		&=\frac{n}{\pi}\sum\limits_{k=1}^n|S_{k,n}(\theta)|\int\limits_{0}^\pi |f(t)-f(\theta_k^{(n)})||S_{k,n}(t)|\ dt+\\
		&\sum\limits_{k=1}^n|f(\theta_k^{(n)})-f(\theta)|S_{k,n}(\theta)|\\
		&=E_1+E_2.
	\end{align*}
	Consider $E_1$.
	Since $f\in C[0,\pi]$, it is uniformly continuous. Therefore, for each $\epsilon>0$, there exists a $\delta>0$, such that $|f(\theta_k^{(n)})-f(t)|<\epsilon$ whenever $|\theta_k^{(n)}-t|\leq \delta $ for all $k=1,2,\ldots, n$.
	Thus, we have 
	\begin{align*}
		E_1&=\frac{n}{\pi}\sum\limits_{k=1}^n|S_{k,n}(\theta)|\int\limits_{0}^\pi |f(t)-f(\theta_k^{(n)})||S_{k,n}(t)|\ dt\\
		&\leq \frac{n}{\pi}\sum\limits_{k=1}^n|S_{k,n}(\theta)|\int\limits_{|t-\theta_k^{(n)}|\leq \delta} |f(t)-f(\theta_k^{(n)})||S_{k,n}(t)|\ dt+\\
		&\quad \frac{n}{\pi}\sum\limits_{k=1}^n|S_{k,n}(\theta)|\int\limits_{|t-\theta_k^{(n)}|>\delta} |f(t)-f(\theta_k^{(n)})||S_{k,n}(t)|\ dt\\
		&\leq \epsilon \frac{n}{\pi}\sum\limits_{k=1}^n|S_{k,n}(\theta)|\int\limits_{|t-\theta_k^{(n)}|\leq \delta} |S_{k,n}(t)|\ dt+\\
		&\qquad \frac{n}{\pi}2\|f\|_\infty\sum\limits_{k=1}^n|S_{k,n}(\theta)|\int\limits_{|t-\theta_k^{(n)}|>\delta} |S_{k,n}(t)|\ dt\\
		&\leq \epsilon \frac{n}{\pi}\sum\limits_{k=1}^n|S_{k,n}(\theta)|\int\limits_{|t-\theta_k^{(n)}|\leq \delta} |S_{k,n}(t)|\ dt+\\
		&\qquad \frac{n}{\pi}2M\sum\limits_{k=1}^n|S_{k,n}(\theta)|\int\limits_{|t-\theta_k^{(n)}|>\delta} |S_{k,n}(t)|\ dt\\
		&\leq C_1\epsilon + \frac{n}{\pi}2M\sum\limits_{k=1}^n|S_{k,n}(\theta)|\int\limits_{|t-\theta_k^{(n)}|>\delta} |S_{k,n}(t)|\ dt,
	\end{align*}
	where $\|f\|_\infty\leq M$ for some $M>0$. 
	
	For $|t-\theta_k^{(n)}|>\delta$ from \cite{grunwald}, we have 
	\begin{align*}
		|S_{k,n}(t)|&\leq \frac{\pi^3}{4n^2}\frac{1}{(\theta-\theta_k^{(n)}-\frac{\pi}{2n})^2}\\
		&\leq \frac{\pi^3}{4n^2}\frac{1}{(\delta-\frac{\pi}{2n})^2}\\
		&=\frac{\pi^3}{4n^2}\frac{4n^2}{(2n\delta-\pi)^2}
		=\mathcal{O}(\frac{1}{n^2}).
	\end{align*}
   Finally, we obtain $E_1\leq C_1\epsilon +C_2 \frac{1}{n}\leq (C_1+C_2)\epsilon,$ for sufficiently large $n$.
	And, we also have 
	\[
	E_2=\sum\limits_{k=1}^n|f(\theta_k^{(n)})-f(\theta)||S_{k,n}(\theta)|\leq C_3\epsilon.
	\]
	Thus, we have $|\mathcal{D}_n(f)(\theta)-f(\theta)|\leq C\epsilon,$ for sufficiently large $n$,
	where $C_4=C_1+C_2+C_3$, since $\epsilon$ is independent of $\theta$, taking supremum over all $\theta\in [0,\pi]$, we have 
	$\|\mathcal{D}_n(f)-f\|_\infty\leq C_4\epsilon$. Since $\epsilon>0$ is arbitrary, we have $\lim\limits_{n\to\infty}\|\mathcal{D}_n(f)-f\|_\infty=0.$
\end{proof}
In what follows, we obtain a quantitative estimate for the convergence of the sequence $\{\mathcal{D}_n\}_{n\in\mathbb{N}}$ in the space $C[0,\pi]$.
\begin{theorem}
	Let $f\in C[0,\pi]$. Then 
	\[
	|\mathcal{D}_n(f)(\theta)-f(\theta)| =\mathcal{O}(\omega(f,\delta_n(\theta))+\omega(f\circ \arccos,\frac{\sqrt{1-x^2}}{n})+\omega(f\circ \arccos, \frac{1}{n^2}),
	\]
	where $x=\cos\theta$.
\end{theorem}
\begin{proof}
	Consider 
	\begin{align*}
		|\mathcal{D}_n(f)(\theta)-f(\theta)|&\leq \frac{n}{\pi}\sum\limits_{k=1}^n|S_{k,n}(\theta)|\int\limits_{0}^\pi |f(t)-f(\theta_k^{(n)})||S_{k,n}(t)|\ dt+\\
		&\qquad \quad \sum\limits_{k=1}^n|f(\theta_k^{(n)})-f(\theta)|S_{k,n}(\theta)|.\\
	\end{align*}
	Let $\delta>0$. Then, using the properties of the modulus of continuity, we have
	\[
	|f(t)-f(\theta_k^{(n)})|\leq (1+|t-\theta_k^{(n)}|\delta^{-1})\omega(f,\delta).
	\]
	Thus, we get
	\begin{align*}
		\frac{n}{\pi}\sum\limits_{k=1}^n&|S_{k,n}(\theta)|\int\limits_{0}^\pi |f(t)-f(\theta_k^{(n)})||S_{k,n}(t)|\ dt\\
		&\leq \omega(f,\delta)\frac{n}{\pi}\sum\limits_{k=1}^n|S_{k,n}(\theta)|\int\limits_{0}^\pi |S_{k,n}(t)|\ dt\\
		&\qquad + \omega(f,\delta)\delta^{-1}\frac{n}{\pi}\sum\limits_{k=1}^n|S_{k,n}(\theta)|\int\limits_{0}^\pi |t-\theta_k^{(n)}||S_{k,n}(t)|\ dt.
	\end{align*}
	Let $\delta=\delta_n(\theta)=\frac{n}{\pi}\sum\limits_{k=1}^n|S_{k,n}(\theta)|\int\limits_{0}^\pi |t-\theta_k^{(n)}||S_{k,n}(t)|\ dt$. 
	Now, 
	\[
	\frac{n}{\pi}\sum\limits_{k=1}^n|S_{k,n}(\theta)|\int\limits_{0}^\pi |f(t)-f(\theta_k^{(n)})||S_{k,n}(t)|\ dt\leq (C_5+1)\omega(f,\delta_n(\theta)),
	\]
	for some constant $C_5>0$.
	By Proposition $3.1$ of \cite{GK}, we have 
	\[
	\int\limits_{0}^\pi |t-\theta_k^{(n)}||S_{k,n}(t)|\ dt=\mathcal{O}(\frac{1+\log n}{n^2}).
	\]
	Therefore, we have $\delta(\theta)=\mathcal{O}(\frac{\log n+1}{n}).$
	We also have 
	\[
	\sum\limits_{k=1}^n|f(\theta_k^{(n)})-f(\theta)|S_{k,n}(\theta)|=\mathcal{O}(\omega(f\circ \cos^{-1},\frac{\sqrt{1-x^2}}{n})+\omega(f\circ \cos^{-1}, \frac{1}{n^2})), 
	\]
	where $x=\cos\theta$. Hence the result.
\end{proof}
\section{Convergence in the space $L^p[0,\pi]$}
In this section, we prove the boundedness and convergence of the sequence of operators $\{\mathcal{D}_n\}_{n\in\mathbb{N}}$ on the space $L^p[0,\pi]$. For $p=\infty$, the case is similar to that of $C[0,\pi]$, by the definition of the operators $\{\mathcal{D}_n\}_{n\in\mathbb{N}}$.
For the other cases, we first prove the uniform boundedness of $\{\mathcal{D}_n\}_{n\in\mathbb{N}}$ on $L^1[0,\pi]$ and then extend this to the space $L^p[0,\pi]$ for $1<p<+\infty$. To prove this, we use the well known Riesz-Thorin interpolation theorem. 

The following lemma establishes the uniform boundedness boundedness of the sequence $\{\mathcal D_n\}_{n\in\mathbb{N}}$ in the space $L^1[0,\pi]$.
\begin{lemma}\label{l1bdd}
		The sequence of operators $\{\mathcal{D}_n:L^1[0,\pi]\to L^1[0,\pi]\}_{n\in\mathbb{N}}$ is bounded uniformly in the space $L^1[0,\pi]$. 
\end{lemma}
\begin{proof}
	From Proposition $3.1$ in \cite{GK}, we get $\int\limits_{0}^\pi|S_{k,n}(\theta)|\ d\theta\leq \frac{\mathcal{C}}{n}$, for some constant $\mathcal C>0$. Therefore, we obtain
	\begin{align*}
		\int\limits_{0}^\pi|\mathcal{D}_n(f)(\theta)|\ d\theta&\leq \frac{n}{\pi}\sum\limits_{k=1}^n\int\limits_{0}^\pi|S_{k,n}(\theta)|\ d\theta\int\limits_{0}^\pi |f(t)||S_{k,n}(t)|\ dt\\
		&\leq \frac{\mathcal Cn}{n\pi}\sum\limits_{k=1}^n \int\limits_{0}^\pi |f(t)||S_{k,n}(t)|\ dt\\
		&=\frac{\mathcal C}{\pi}\int\limits_{0}^\pi|f(t)|\sum\limits_{k=1}^n|S_{k,n}(t)|\ dt\\
		&\leq \frac{\mathcal Cc_1}{\pi}\int\limits_{0}^\pi|f(t)|\ dt.
	\end{align*}
	Let $\tilde C=\frac{\mathcal Cc_1}{\pi}$. Therefore, we have 
	$\|\mathcal{D}_n\|_1\leq \tilde C\|f\|_1$ for all $f\in L^1[0,\pi]$, where $\tilde C$ is independent of $n$ and $f$.
\end{proof}
Now, we prove the uniform boundedness of the sequence operators $\{\mathcal{D}_n\}_{n\in\mathbb{N}}$ on the space $L^p[0,\pi]$. To prove this, we recall a fundamental interpolation theorem in harmonic analysis, namely the Riesz--Thorin interpolation theorem (see Theorem 1.19, \cite{fourier}), stated below.
\begin{theorem}[Riesz--Thorin Interpolation]\label{RT}
	Let $1\leq p_0,p_1,q_0,q_1\leq +\infty$, and for $0<\theta<1$, define $p$ and $q$ by 
	\begin{equation}\label{con}
		\frac{1}{p}=\frac{1-\theta}{p_0}+\frac{\theta}{p_1},\quad 	\frac{1}{q}=\frac{1-\theta}{q_0}+\frac{\theta}{q_1}.
	\end{equation}
	If $T$ is a linear operator from $L^{p_0}+L^{p_1}
	$ to $L^{q_0}+L^{q_1}$ such that 
	\[
	\|Tf\|_{q_0}\leq M_0\|f\|_{p_0}, \ for\ f\in L^{p_0}
	\]
	and 
	\[
	\|Tf\|_{q_1}\leq M_1\|f\|_{p_1}, \ for\ f\in L^{p_1}.
	\]
	Then 
	\[
	\|Tf\|_{q}\leq M_0^{1-\theta}M_1^\theta\|f\|_{p}, \ for\ f\in L^{p}
	\]
\end{theorem}
Now, we have the following result. 
\begin{lemma}\label{lpbdd}
	The sequence of operators $\{\mathcal{GK}_n:L^p[0,\pi]\to L^p[0,\pi]\}_{n\in\mathbb{N}}$ is bounded uniformly in the space $L^p[0,\pi]$ for $1<p<+\infty$.
\end{lemma}
\begin{proof}
	From Lemma \ref{bdd} and \ref{l1bdd}, we have  
	\[
	\|\mathcal{D}_n(f)\|_\infty\leq C\|f\|_\infty, \ for \ f\in L^\infty[0,\pi] \ for\ all\ n\in\mathbb{N}
	\]
	and 
	\[
	\|\mathcal{D}_n(f)\|_1\leq \tilde C\|f\|_1, \ for \ f\in L^1[0,\pi]\ for\ all\ n\in\mathbb{N}.
	\]
	
	Therefore, by applying Theorem \ref{RT}, we have 
	\[
	\|\mathcal{D}_n(f)\|_p\leq C^{1-\theta}\tilde C^\theta\|f\|_p, \ for \ f\in L^p[0,\pi]\ for\ all\ n\in\mathbb{N}.
	\]
	for $\theta\in (0,1)$, $p_0=q_0=1$, $p_1=q_1=\infty$ and $p$, $q$ satisfying \ref{con}. In other words, we see that for $1<p<+\infty$, 
	\[
	\|\mathcal{D}_n(f)\|_p\leq C_p\|f\|_p, \ for \ f\in L^p[0,\pi]\ for\ all\ n\in\mathbb{N},
	\]
	where $C_p$ does not depend on $n$ or $f$.
\end{proof}
\subsection{Convergence in the space $L^p[0,\pi]$ using a Korovkin-type Theorem}

Korovkin-type theorems are fundamental in approximation theory, as they provide a simple and effective criterion for establishing the convergence of sequences of positive linear operators. It is sufficient to verify convergence on just three polynomial test functions in order to guarantee convergence on the entire space \cite{korovkin, altomare}. This remarkable result unifies several classical approximation theorems in the literature, including those due to Bernstein, Weierstrass, and Fej\'er. Over time, these results have been extended to various functional settings, including Banach function spaces.

In this section, we recall some fundamental definitions concerning Banach function spaces and present a Korovkin-type theorem in this framework. Using this theorem, we extend the convergence results for $\{\mathcal{D}_n\}_{n\in\mathbb{N}}$ obtained in $L^p[0,\pi]$ to a broader class of Banach function spaces.

Let $(A, \mathcal{S}, \mu)$ be a measurable space, where $\mathcal{S}$ is a $\sigma$-algebra of measurable subsets of $A$. Let $\mathcal{M}$ denote the set of all measurable functions on $A$, and let $\mathcal{M}^+$ denote the set of all non-negative measurable functions on $A$.

First, we recall the definition of a function norm. 
\begin{definition}
	A mapping $\rho: \mathcal{M}^+ \to [0, +\infty]$ is called a \emph{function norm} if the following properties hold for all $f, g, f_n \in \mathcal{M}^+$, $a \geq 0$, and $E \in \mathcal{S}$:
	\begin{enumerate}
		\item $\rho(f) = 0$ if and only if $f = 0$ $\mu$-a.e.; $\rho(af) = a \rho(f)$; and $\rho(f+g) \leq \rho(f) + \rho(g)$;
		\item If $g \leq f$ $\mu$-a.e., then $\rho(g) \leq \rho(f)$;
		\item If $f_n \uparrow f$ $\mu$-a.e., then $\rho(f_n) \uparrow \rho(f)$;
		\item If $\mu(E) < +\infty$, then $\rho(\chi_E) < +\infty$;
		\item If $\mu(E) < +\infty$, then there exists a constant $C_E > 0$ such that
		\[
		\int_E f \, d\mu \leq C_E \rho(f),
		\]
		where $C_E$ depends on $E$ and $\rho$, but not on $f$.
	\end{enumerate}
\end{definition}

A Banach function space $X$ generated by $\rho$ is a Banach space consisting of functions $f \in \mathcal{M}$ equipped with the norm $\|f\|_X := \rho(|f|)$.

\begin{definition}
	A function $f \in X$ is said to have an \emph{absolutely continuous norm} if
	\[
	\|f \chi_{E_n}\|_X \to 0 \quad \text{as } n \to \infty,
	\]
	for every sequence $\{E_n\} \subset \mathcal{S}$ such that $E_n \to \emptyset$ $\mu$-a.e. (i.e., $\chi_{E_n} \to 0$ $\mu$-a.e.).
\end{definition}

Define
\[
X_a = \{ f \in X : f \text{ has an absolutely continuous norm} \}.
\]

\textbf{Special case:} Let $A = [a,b]$, $\mathcal{S}$ be the Borel $\sigma$-algebra, and $\mu$ be the Lebesgue measure.

We now focus on a particular case of Banach function spaces. Without loss of generality, let $A = [0,1]$, $\mu$ be the Lebesgue measure, and $\mathcal{S}$ the Borel $\sigma$-algebra. Let $\rho$ be a function norm on this measure space, and let $X$ be the corresponding Banach function space with norm $\|f\|_X = \rho(|f|)$.

For $\delta > 0$, define the shift operator $T_\delta$ on $X$ by
\[
T_\delta(f)(x) =
\begin{cases}
	f(x+\delta), & \text{if } x+\delta \in [0,1], \\
	0, & \text{otherwise}.
\end{cases}
\]

Let $X^S$ denote the closure in $X$ of the set
\[
\left\{ f \in X : \lim_{\delta \to 0} \|T_\delta(f) - f\|_X = 0 \right\}.
\]
It is known that $X^S$ includes $C[0,1]$ and, moreover, properly contains it (see \cite{vinaya, zeren}).

In 2022, Zeren et al.~\cite{zeren} proved a Korovkin-type theorem on the subspace $X^S$ and in 2025 Kiran Kumar and Vinaya further generalized this result to a non-positive setting, as stated below.

Let $B(X^S)$ denote the space of all bounded linear operators on $X^S$.

\begin{theorem}\cite{vinaya}\label{nonpositive}
	Let $X$ be a Banach function space such that $1 \in X_a$, and let $\{L_n\}_{n \in \mathbb{N}}$ be a sequence of bounded linear operators on $X^S$ satisfying:
	\begin{enumerate}
		\item $\lim\limits_{n \to \infty} L_n(g) = g$ in $C[0,1]$ for all $g \in \{1, t, t^2\}$;
		\item $\sup\limits_{n} \|L_n\|_{B(C[0,1])} < \infty$.
	\end{enumerate}
	Then $\lim\limits_{n \to \infty} L_n(f) = f$ in $X$ for all $f \in X^S$ if and only if	$\sup\limits_{n} \|L_n\|_{B(X^S)} < \infty.$
\end{theorem}
We apply Theorem \ref{nonpositive} to the sequence of operators $\{\mathcal D_n\}_{n\in\mathbb{N}}$ on the Banach function space, $L^p[0,\pi]$, since these operators are non-positive as already mentioned. We have the following result.
\begin{corollary}
	The sequence of operators $\mathcal{D}_n:L^p[0,\pi]\to L^p[0,\pi]$, for $1\leq p<+\infty$ satisfies the following convergence, that is
\[
\lim\limits_{n\to\infty}\|\mathcal{D}_n(f)-f\|_p=0\  for\  all\ f \in L^p[0,\pi].
\]
	\end{corollary}
	\begin{proof}
Using Lemma \ref{bdd} and \ref{main}, the sequence of operators $\{\mathcal{D}_n\}_{n\in\mathbb{N}}$ defined on the Banach function space $L^p[0,\pi]$, for $1\leq p<+\infty$ satisfies the conditions $(1)$ and $(2)$ of Theorem \ref{nonpositive}. The uniform boundedness of the sequence on $L^p[0,\pi]$ for $1\leq p<+\infty$ has been established in Lemma \ref{l1bdd} and \ref{lpbdd}. Hence using Theorem \ref{nonpositive}, we conclude that 
\[
\lim\limits_{n\to\infty}\|\mathcal{D}_n(f)-f\|_p=0\ for\  all\ f \in L^p[0,\pi]^S.
\]
We have $C[0,\pi]\subseteq L^p[0,\pi]^S\subseteq L^p[0,\pi]$. Since $L^p[0,\pi]^S$ is closed and $C[0,\pi]$ is dense in $L^p[0,\pi]$, we must have $L^p[0,\pi]^S=L^p[0,\pi]$. Hence the convergence holds for all $f\in L^p[0,\pi]$. 
\end{proof}
\subsection{Quantitative Estimates in $L^p[0,\pi]$}
	In this section, we establish a quantitative estimate for the convergence of the Gr\"unwald--Durrmeyer operators in the space $L^p[0,\pi]$ for $1\leq p<+\infty$.  We consider the following version of Petree's K-functional, in which the derivative is measured in the uniform norm unlike the $L^p$- norm.

For $1\leq p<+\infty$, $f \in L^p[0.\pi]$ and $\delta > 0$, we consider
\[
K_p(f,\delta)
=
\inf_{g\in C^1[0,\pi]}
\left\{
\|f-g\|_{p}
+
\delta \|g'\|_{\infty}
\right\},
\]
$C^1[0,\pi]$ denotes the space of all continuously differentiable functions on $[0,\pi]$. 
The following theorem establishes a quantitative estimate for the convergence of $\{\mathcal{D}_n\}_{n\in\mathbb{N}}$ using this K-functional.
\begin{theorem}
	Let $1\leq p<+\infty$ and $f\in L^p[0,\pi]$, then
\[ 
\|\mathcal{D}_n(f)-f\|_p\leq R_pK_p(f,\tilde R_pm_{n}),
\]
for some positive constants $R_p$ and $\tilde R_p$ and 
\begin{align*}
m_{n}=\begin{cases}
	\frac{1+\log n}{n}, & p = 1, \\[8pt]
	\frac{1+\log n}{n}+\frac{1}{n^\frac{1}{p}}, & 1<p<+\infty.
\end{cases}
\end{align*}
\end{theorem}
\begin{proof}
		Let $f\in L^p[0,\pi]$, $g\in C^1[0,\pi]$ and $\theta\in [0,\pi]$. Then, 
	\[
	\mathcal{D}_n(f)(\theta)-f(\theta)=\mathcal{D}_n(f-g)(\theta)+\mathcal{D}_n(g)(\theta)-g(\theta)+g(\theta)-f(\theta).
	\]
	Therefore, we have 
	\begin{align*}
		\|\mathcal{D}_n(f)-f\|_p&\leq \|\mathcal{D}_n(f-g)\|_p+\|\mathcal{GK}_n(g)-g\|_p+\|g-f\|_p.\\
		&\leq
		C_p\|f-g\|_p+\|\mathcal{D}_n(g)-g\|_p+\|g-f\|_p\\ 
		&\leq (C_p+1)\|f-g\|_p+\|\mathcal{D}_n(g)-g\|_p\\
		&=(C_p+1)\|f-g\|_p+\|\mathcal{D}_n(g)-g\|_p, 
	\end{align*}
	since $\|\mathcal{}_n(f)\|_p\leq C_p\|f\|_p$ for some constant $C_p>0$.
	
	Applying mean value theorem, for $t,\theta\in [0,\pi]$, there exists a $\xi$ between $t$ and $\theta$ such that $|g(t)-g(\theta)|=|g'(\xi)||t-\theta|\leq \|g'\|_\infty|t-\theta|$. Therefore, we have 
	\begin{align*}
		|\mathcal{D}_n(g)(\theta)-g(\theta)|&\leq \frac{n}{\pi}\sum\limits_{k=1}^n|S_{k,n}(\theta)|\int\limits_{0}^\pi |g(t)-g(\theta)||S_{k,n}(t)|\ dt\\
		&\leq \|g'\|_\infty\frac{n}{\pi}\sum\limits_{k=1}^n|S_{k,n}(\theta)|\int\limits_{0}^\pi |t-\theta||S_{k,n}(t)|\ dt\\
		&\leq \|g'\|_\infty\frac{n}{\pi}\sum\limits_{k=1}^n|S_{k,n}(\theta)|\int\limits_{0}^\pi |t-\theta_k^{(n)}||S_{k,n}(t)|\ dt+\\
		&\quad \|g'\|_\infty\frac{n}{\pi}\sum\limits_{k=1}^n|\theta-\theta_k^{(n)}||S_{k,n}(\theta)|\int\limits_{0}^\pi |S_{k,n}(t)|\ dt.
	\end{align*}
	The first expression becomes
	\begin{align*}
		\|g'\|_\infty\frac{n}{\pi}\sum\limits_{k=1}^n|\theta-\theta_k^{(n)}||S_{k,n}(\theta)|\int\limits_{0}^\pi |S_{k,n}(t)|\ dt&\leq C_6\|g'\|_\infty n\left (\frac{1+\log n}{n^2}\right )\sum\limits_{k=1}^n|S_{k,n}(\theta)|\\
		&\leq C_6c_1\|g'\|_\infty\left (\frac{1+\log n}{n}\right ).
		\end{align*} 
	
	Applying Minkowski's inequality, we get 
	\begin{align*}
		\|\mathcal{D}_n(g)-g\|_p&\leq C_6c_1\|g'\|_\infty\left (\frac{1+\log n}{n}\right )+\\
		&\quad \|g'\|_\infty \frac{n}{\pi}\sum\limits_{k=1}^n\left (\int\limits_{0}^\pi|\theta-\theta_k^{(n)}|^p|S_{k,n}(\theta)|^p\ d\theta \right )^\frac{1}{p}\int\limits_{0}^\pi |S_{k,n}(t)|\ dt.
	\end{align*}
	By Proposition $3.1$ of \cite{GK}, we get 
		\begin{enumerate}
		\item[(i)] \[
		\left (\int\limits_{0}^{\pi} |\theta-\theta_k^{(n)}|^p \, |S_{k,n}(\theta)|^p \, d\theta\right )^\frac{1}{p}
		=
		\begin{cases}
			\mathcal{O}\!\left(\dfrac{\log n + 1}{n^2}\right), & p = 1, \\[8pt]
			\mathcal{O}\!\left(\dfrac{1}{n^{1+\frac{1}{p}}}\right), & 1<p<+\infty.
		\end{cases}
		\]
		\item[(ii)]
		\[
		\left (\int\limits_{0}^{\pi} |S_{k,n}(\theta)|^p \, d\theta\right )^\frac{1}{p}
		=\mathcal{O}\left (\frac{1}{n^\frac{1}{p}}\right ), \quad 1\leq p<+\infty.
		\			\]
	\end{enumerate}
	Now, we have the following two cases.
	
\textbf{Case 1: $p=1$} From $(i)$ and $(ii)$, we get 
\[
\|\mathcal{D}_n(g)-g\|_1\leq \tilde C_{1}\|g'\|_\infty\left (\frac{1+\log n}{n}\right ).
\]
\textbf{Case 2: $1<p<+\infty$} Again, from $(i)$ and $(ii)$, we obtain
\[
\|\mathcal{D}_n(g)-g\|_p\leq \tilde C_{p}\|g'\|_\infty\left (\frac{1+\log n}{n}+\frac{1}{n^\frac{1}{p}}\right ),
\]
where $C_{1,p}>0$ and $C_{2,p}>0$ are constants. 
Let 
\[
m_{n}=\begin{cases}
	\frac{1+\log n}{n}, & p = 1, \\[8pt]
	\frac{1+\log n}{n}+\frac{1}{n^\frac{1}{p}}, & 1<p<+\infty.
\end{cases}
\]
	\begin{align*}
	\|\mathcal{D}_n(f)-f\|_p&
	\leq (C_p+1)\|f-g\|_p+\|\mathcal{D}_n(g)-g\|_p\\
	&\leq (C_p+1)\|f-g\|_p+\tilde C_pm_{n}\|g'\|_\infty\\
	&=(C_p+1)\{\|f-g\|_1+\frac{\tilde C_p}{C_p+1}m_{n}\|g'\|_\infty\}.
\end{align*}
	Now taking infimum over $g\in C^1[0,\pi]$, we have
\[ 
\|\mathcal{D}_n(f)-f\|_p\leq \tilde C_pK_p(f,\tilde R_pm_{n}),
\]
where $R_p=C_p+1$ and $\tilde R_p=\frac{\tilde C_p}{(C_p+1)}$.
\end{proof}
\section*{Concluding Remarks}
In this paper, we constructed a Durrmeyer-type variant of the classical Gr\"unwald interpolation operators on the space $L^p[0,\pi]$ for $1\leq p\leq+\infty$. We established the boundedness and convergence of these operators in this space. Moreover, we obtained quantitative estimates for the convergence of these operators using the modulus of continuity and a K-functional. Extension of these convergence results and the corresponding quantitative estimates on more general Banach function spaces remains an open problem which will be addressed in the future.
\section*{Acknowledgments}
The author wishes to thank Prof. M. N. N. Namboodiri and Dr. V. B. Kiran Kumar for fruitful discussions on related topics.

\end{document}